\documentclass[a4paper,  leqno,  myheadings,  11pt,  twoside]{article}
\usepackage{mathrsfs}
\usepackage{hyperref}

\usepackage{amsmath,  amsthm,  amssymb,  amscd,  amsxtra, graphicx}
\usepackage{latexsym,  amsfonts}

\DeclareMathOperator*{\esssup}{esssup}

\newcommand{\lnorm}[1]{{\lVert#1\rVert}_\infty}
\def\qc{quasi\-conformal }

\def\md{\mathbb{D}}
\def\mc{\mathbb{C}}
\def\msc{\mathscr{C}}
\def\mcc{\mathcal{C}}

\def\mn{\mathbb{N}}

\def\T{Teich\-m\"ul\-ler }

\def\nmu{\|\mu\|_\infty}

\def\nmu{\|\mu\|_\infty}

\def\wt{\widetilde}
\def\vp{\varphi}

\def\vp{\varphi}

\def\pa{\partial}

\def\bs{B(S)}
\def\ov{\overline}

\def\de{\md}
\def\q1s{Q^1(S)}

\def\ts{T(S)}

\def\qmd{Q(\md)}

\def\tmd{T(\md)}

\def\mcs{\mathcal{S}}

\def\emu{[\mu]}

\def\emb{[\mu]_B}

 \makeatletter
\@addtoreset{equation}{section}

\newtheorem{theorem}{Theorem}[section]

\newtheorem{lemma}{Lemma}[section]

\newtheorem{theo}{Theorem}

\newtheorem{defn}{Definition}

\begin{document}

\title{\bf{Non-decreasable extremal Beltrami differentials of non-landslide type}
\author{GUOWU YAO
\\
Department of Mathematical Sciences,
  Tsinghua University\\ Beijing,  100084,
   People's Republic of
  China \\E-mail: \texttt{gwyao@math.tsinghua.edu.cn}
}}
 \date{April 19,  2015}
\maketitle
\begin{abstract}\noindent
In this paper, we deform a uniquely-extremal Beltrami differential into different non-decreasable Beltrami differentials,   and then  construct  non-unique extremal  Beltrami differentials such that they are both non-landslide  and  non-decreasable.

\end{abstract}
\renewcommand{\thefootnote}{}

\footnote{
\qc mapping, non-decreasable,  non-landslide, locally extremal.}

\footnote{The  work was  supported by   the National Natural Science Foundation of China (Grant
No. 11271216).}

\section{\!\!\!\!\!{. }
Introduction}\label{S:intr}

 \indent Let $S$ be a plane domain with at least two boundary points.  The \T space $\ts$ is the space of
equivalence classes of \qc maps $f$ from $S$ to a variable domain
$f(S)$. Two \qc maps $f$ from $S$ to $f(S)$ and $g$ from $S$ to
$g(S)$ are said to  be equivalent, denoted by $f\sim g$,  if there is a conformal map $c$ from $f(S)$
onto $g(S)$ and a homotopy through \qc maps $h_t$ mapping $S$ onto
$g(S)$ such that $h_0=c\circ f$, $h_1=g$ and $h_t(p)=c\circ
f(p)=g(p)$ for every $t\in [0,1]$ and every $p$ in the boundary of
$S$. Denote by $[f]$  the \T equivalence class of $f$; also
sometimes denote the equivalence class by $\emu$ where $\mu$ is the
Beltrami differential of $f$.

Denote by $Bel(S)$ the Banach space of Beltrami differentials
$\mu=\mu(z)d\bar z/dz$ on $S$ with finite $L^{\infty}$-norm and by
$M(S)$ the open unit ball in $Bel(S)$.

For $\mu\in M(S)$, define
\begin{equation*}
k_0(\emu)=\inf\{\|\nu\|_\infty:\,\nu\in\emu\}.
\end{equation*}
We say that $\mu$ is extremal  in $\emu$
  if $\nmu=k_0(\emu)$ (the corresponding \qc map $f$ is said to be extremal for its boundary values as well), uniquely extremal if $\|\nu\|_\infty>k_0(\mu)$ for any other
$\nu\in\emu$.

The cotangent space to $\ts$ at the basepoint is the Banach space
$Q(S)$ of integrable holomorphic quadratic differentials on $S$ with
$L^1-$norm
\begin{equation*}
\|\vp\|=\iint_{S}|\vp(z)|\, dxdy<\infty.
 \end{equation*}  In what follows,  let $\q1s$
denote the unit sphere of $Q(S)$.

 As is well
known, $\mu$ is extremal if and only if
  it  has a  so-called Hamilton sequence, namely, a sequence
$\{\psi_n\}\subset \q1s$, such that
\begin{equation}
\lim_{n\to\infty}Re\iint_S \mu\psi_n(z)dxdy=\nmu.
\end{equation}
By definition, a sequence $\{\psi_n\}$ is called degenerating if it
converges to 0 uniformly on compact subsets of $S$.

We would not like to give the exact definition of  Strebel point and non-Strebel point in $\ts$.  But it should be kept in mind that an  extremal  represents a non-Strebel point  if and only if it   has a degenerating Hamilton sequence (for example,  see \cite{ELi,GL}). We call an extremal representing a non-Strebel point to be a non-Strebel extremal.

\begin{defn}An extremal Beltrami differential $\mu$ in $Bel(S)$ is  said to be
of landslide type if there exists a non-empty open subset $E\subset S$ such
that
\begin{equation*}\esssup_{z\in
E}|\mu(z)|<\|\mu\|_\infty;\end{equation*} otherwise, $\mu$ is said
to be of non-landslide type.
\end{defn}
The conception of \emph{non-landslide} was firstly introduced by Li in \cite{Li5}. It was proved by Fan \cite{Fan} and the author \cite{Yao7} independently that if $\mu$ contains more than one extremal, then it contains infinitely many extremals of non-landslide type.

The following notion of \emph{locally extremal} was  introduced in \cite{Shr} by Sheretov.
\begin{defn}
 A Beltrami differential $\mu$ in $M(S)$ is called to be
locally extremal if for any domain $G\subset S$ it is extremal
in its class in $T(G)$; in other words,
\begin{equation*}
\|\mu\|_G:=\esssup_{z\in
G}|\mu|=\sup\{\frac{Re\iint_G\mu\phi(z)dxdy}{\iint_G|\phi(z)|dxdy}:\,\phi\in
Q^1(G)\}.
\end{equation*}

 \end{defn}
Obviously, the extremality for $\mu$  in $S$ is a prerequisite condition  for $\mu$ to be locally extremal. However, up to present,  it is not clear whether a \T class always contains a local extremal.

  \begin{defn}
 A Beltrami differential $\mu$ (not necessarily extremal) is called to be  non-decreasable in its class $\emu$ if  for $\nu\in \emu$,
  \begin{equation}
  |\nu(z)|\leq|\mu(z)|\  a.e. \text{\ in } S,\end{equation}
implies that $\mu=\nu$; otherwise, $\mu$  is called to be decreasable.

 \end{defn}
 The notion of non-decreasable dilatation was firstly introduced by Reich in \cite{Re4} when he studied the unique extremality of \qc mappings. The author \cite{Yao2} proved that the non-decreasable extremal in a class may be non-unique.   Shen and Chen  \cite{SC} proved that there are infinitely many non-decreasable representatives (generally, not extremal) in a class while the existence of a non-decreasable extremal is generally unknown. It should be noted that a non-unique extremal is certainly of non-constant modulus if it is non-decreasable.

 In particular, a unique extremal is
naturally   non-landslide, locally extremal and non-decreasable. However, it is not clear what about the converse. The following problem is posed in \cite{Yao7}.

\noindent\textbf{Problem $\mathscr{A}$}.
Is there a non-unique extremal $\mu$ such that $\mu$ is non-landslide, locally extremal and  non-decreasable?

Up to present, the problem seems open. We can find some examples related to the problem in literatures. The first (even essentially only) example of
 non-unique extremal   which is both non-landslide and locally extremal was given by Reich  in \cite{Re3},  but the extremal is  decreasable for it has a constant modulus.  To get a non-unique extremal of non-constant modulus that is both non-landslide and locally extremal, one may apply the Construction Theorem in \cite{Yao3} in a refined manner.
 The second example  was given by the author in Theorem 1 (2) of \cite{Yao2} that  provides a non-unique extremal which is both  locally extremal and  non-decreasable but landslide.

 One might expect the third example for a non-unique extremal which is both  non-landslide and  non-decreasable.  However, no such an extremal can be found in literatures. The motivation of this paper is to construct such an example.

 This paper is organized as follows. In Section \ref{S:prepar}, we introduce the Main Inequality and give its application.  The Infinitesimal Main Inequality is introduced in Section \ref{S:infprepar}.  We construct  extremals  which are both  non-landslide and  non-decreasable in Section  \ref{S:appl}.
The Construction Theorem for the desired extremals is obtained  in the last Section \ref{S:construction}.

\section{\!\!\!\!\!{. }
 Main Inequality and its application}\label{S:prepar}

For brevity, we restrict our consideration to $S=\md$, where  $\md=\{z\in \mc:\;|z|<1\}$ is the unit disk. The universal \T space $\tmd$ can be viewed as the set of  the equivalence classes $[f]$ of \qc mappings $f$ from $\de$ onto itself. So, for any \qc mapping from
$\de$ onto itself, there is no difference between $\tmd$ and $T(f(\md))$.

The Reich-Strebel inequality, so-called Main Inequality (see \cite{Gar, RS1, RS3}), plays an important role in the study of \T theory.
To introduce the inequality, we need some denotations. Suppose that $f$ and $g$ are two \qc mappings of $\de$ onto itself with the Beltrami differentials $\mu,\;\nu$ respectively.  Let $F=f^{-1}, \; G=g^{-1}$ and $\wt \mu,\;\wt \nu$ denote the Beltrami differentials of $F,\;G$ respectively. Put
$\alpha=\wt \mu\circ f, \;\beta=\wt\nu\circ f$. Then we have\\
\textbf{Main Inequality.} If $\mu\sim \nu$, i.e., $f$ and $g$ are equivalent, then for any $\vp\in \qmd$,
\begin{equation}\label{Eq:main}
\iint_\md\vp\;dxdy \leq \iint_\md
|\vp(z)|\frac{\left|1-\mu(z)\frac{\vp(z)}{|\vp(z)|}\right|^2}{1-|\mu(z)|^2}\frac{\left|1+ \beta\frac{\mu}{\alpha} \frac{1-\ov\mu\frac{\ov\vp}{|\vp|}}{ 1-\mu\frac{\vp}{|\vp|}}   \right|^2}{1-|\beta|^2}    \;dxdy,
\end{equation}
or equivalently (see \cite{Re3, Re4}),
\begin{equation}\label{Eq:maineq}
Re\iint_\md\frac{(\beta-\alpha)(1-\alpha\ov \beta)\tau} {(1-|\alpha|^2)(1-|\beta|^2)}\vp \;dxdy
\leq\iint_\md \frac{|\alpha-\beta|^2}{(1-|\alpha|^2)(1-|\beta|^2)}|\vp|\;dxdy,
\end{equation}
where $\tau=\frac{\ov{\pa_z f}}{\pa_z f}=-\frac{\mu}{\alpha}$.

Let
\begin{align*}
sgn z=\begin{cases}\frac{z}{|z|},\;&z\neq 0,\\
0,\;&z=0\end{cases}
\end{align*}
be the signal function of $z\in \mc$.
\begin{lemma}\label{Th:decr}
With the same notations as above, if $\mu\sim \nu$ and $|\wt \nu(w)|\leq |\wt \mu(w)|$ for almost every $w\in \md$, then there is a constant $C$ depending only on $k=\lnorm{\mu}$, such that for any $\vp\in \qmd$,
\begin{equation}\label{Eq:decr1}
\iint_\Lambda|\alpha-\beta|^2|\vp | \;dxdy
\leq C\iint_\Lambda [|\vp|-Re(\vp sgn \mu)]\;dxdy,
\end{equation}
where $\Lambda=\{z\in \md:\; \mu(z)\neq 0\}$.
\end{lemma}
\begin{proof}
 Since $\alpha(z)=\wt\mu(f(z))=-\frac{\mu(z)}{\tau}$, we have $\alpha(z)=0$ when  $z\in \md\backslash \Lambda$. By the condition
$|\wt \nu(w)|\leq |\wt \mu(w)|$ ($w=f(z)\in \md$), it forces $\beta(z)=0$ when  $z\in \md\backslash \Lambda$. It follows directly from (\ref{Eq:maineq})  that
\begin{equation}\label{Eq:decr2}
\begin{split}
&-\iint_\Lambda \frac{|\alpha-\beta|^2}{(1-|\alpha|^2)(1-|\beta|^2)}|\vp|\;dxdy\leq Re\iint_\Lambda\frac{(\beta-\alpha)(1-\alpha\ov \beta)} {(1-|\alpha|^2)(1-|\beta|^2)}\frac{\mu}{\alpha}\vp \;dxdy\\
&=-Re\iint_\Lambda\frac{(\alpha-\beta)(1-\alpha\ov \beta)} {(1-|\alpha|^2)(1-|\beta|^2)}\frac{|\alpha|}{\alpha}\vp sgn \mu \;dxdy.
\end{split}
\end{equation}
In order to group $|\vp|-\vp sgn \mu$ together, we add
\[Re\iint_\Lambda\frac{(\alpha-\beta)(1-\alpha\ov \beta)} {(1-|\alpha|^2)(1-|\beta|^2)}\frac{|\alpha|}{\alpha}|\vp| \;dxdy\]
to both sides of (\ref{Eq:decr2}) and get
\begin{equation}\label{Eq:decr3}
\begin{split}
&Re\iint_\Lambda\frac{(\alpha-\beta)(1-\alpha\ov \beta)} {(1-|\alpha|^2)(1-|\beta|^2)}\frac{|\alpha|}{\alpha}|\vp| \;dxdy-\iint_\Lambda \frac{|\alpha-\beta|^2}{(1-|\alpha|^2)(1-|\beta|^2)}|\vp|\;dxdy\\
&\leq Re\iint_\Lambda\frac{(\alpha-\beta)(1-\alpha\ov \beta)} {(1-|\alpha|^2)(1-|\beta|^2)}\frac{|\alpha|}{\alpha}(|\vp|-\vp sgn \mu) \;dxdy.
\end{split}
\end{equation}
By a deformation, we have
\begin{equation}\label{Eq:decr4}
\begin{split}
&\iint_\Lambda\frac {(1-|\alpha|)|\alpha-\beta|^2+(1+|\alpha|)(|\alpha|^2-|\beta|^2)}{2|\alpha|(1+|\alpha|)(1-|\beta|^2)}|\vp| \;dxdy\\
&\leq Re\iint_\Lambda\frac{(\alpha-\beta)(1-\alpha\ov \beta)} {(1-|\alpha|^2)(1-|\beta|^2)}\frac{|\alpha|}{\alpha}(|\vp|-\vp sgn \mu) \;dxdy.
\end{split}
\end{equation}
Then,
\begin{equation}\label{Eq:decrea}
\begin{split}
&\iint_\Lambda\frac {(1-|\alpha|)|\alpha-\beta|^2}{2|\alpha|(1+|\alpha|)(1-|\beta|^2)}|\vp| \;dxdy\\
&\leq Re\iint_\Lambda\frac{(\alpha-\beta)(1-\alpha\ov \beta)} {(1-|\alpha|^2)(1-|\beta|^2)}\frac{|\alpha|}{\alpha}(|\vp|-\vp sgn \mu) \;dxdy.
\end{split}
\end{equation}
Since $|\beta(z)|\leq|\alpha(z)|\leq k$, one finds that a lower bound on the coefficient of $|\vp|$ on the left of (\ref{Eq:decr4}) is
\begin{align*}
&\frac {(1-|\alpha|)|\alpha-\beta|^2}{2|\alpha|(1+|\alpha|)(1-|\beta|^2)}\geq \frac{1-k}{2k(1+k)}|\alpha-\beta|^2.
\end{align*}
 An upper bound for the integrand on the right of (\ref{Eq:decr4}) is
 \begin{align*}
&|\alpha-\beta|\frac{1+|\alpha|^2}{(1-|\alpha|^2)(1-|\beta|^2)}||\vp|-\vp sgn \mu|\leq \frac{1+k^2}{(1-k^2)^2}|\alpha-\beta|\cdot||\vp|-\vp sgn \mu|.
\end{align*}
Therefore, by the  the identity
\begin{equation}\label{Eq:identity}||w|-w|^2=2|w|(|w|-Re w),\end{equation}
we have
\begin{equation}\label{Eq:decr5}
\begin{split}
\iint_\Lambda |\alpha-\beta|^2|\vp|\;dxdy&\leq \frac{2k(1+k^2)}{(1+k)(1-k)^3}\iint_\Lambda |\alpha-\beta|\cdot||\vp|-\vp sgn \mu|\;dxdy\\
&=C'\iint_\Lambda |\alpha-\beta||\vp|^{\frac{1}{2}}[|\vp|-Re(\vp sgn \mu)]^{\frac{1}{2}}\;dxdy,
\end{split}
\end{equation}
where $C'=C'(k)=\frac{2\sqrt{2}k(1+k^2)}{(1+k)(1-k)^3}.$

Applying Schwarz's Inequality, we get
 \begin{align*}
&\iint_\Lambda |\alpha-\beta|^2|\vp|\;dxdy\leq C'^2\iint_\Lambda  [|\vp|-Re(\vp sgn \mu)]\;dxdy.
\end{align*}

\end{proof}

\section{\!\!\!\!\!{. }
Infinitesimal Main Inequality}\label{S:infprepar}

Two Beltrami differentials $\mu$ and $\nu$ in $Bel(S)$ are said to
be infinitesimally equivalent, denoted by $\mu\approx\nu$,  if
\begin{equation*}\int_S\mu\vp =\int_S \nu\vp,  \text{
for any } \vp\in Q(S).
\end{equation*}
The tangent space $\bs$ of $\ts$ at the basepoint is defined as the
set of the quotient space of $Bel(S)$ under the equivalence
relation.  Denote by $\emb$ the equivalence class of $\mu$ in
$\bs$.  The  set of all  Beltrami differentials equivalent to zero is
called the $\mathscr{N}-$class in $Bel(S)$.

 We say that $\mu$ is (infinitesimally) extremal  (in $\emb$) if $\nmu=\|\emb\|$,
(infinitesimally) uniquely extremal if $\|\nu\|_\infty>\nmu$ for any other $\nu\in
\emb$.

The notion of infinitesimal Strebel point and non-Strebel point can be found in \cite{ELi}. Any extremal in an infinitesimal non-Strebel point is called an infinitesimal non-Strebel extremal.

\begin{defn}
 A Beltrami differential $\mu$ (not necessarily extremal) is called to be infinitesimally non-decreasable in its class $\emb$ if  for $\nu\in \emb$,
  \begin{equation}
  |\nu(z)|\leq|\mu(z)|\  a.e. \text{\ in } S,\end{equation}
implies that $\mu=\nu$; otherwise, $\mu$  is called to be infinitesimally decreasable.
 \end{defn}

 The following is the Infinitesimal Main Inequality on $\md$, whose proof can be found in \cite{BLMM,Re3}. \\
 \textbf{Infinitesimal Main Inequality.} Suppose $\mu,\;\nu\in M(\md)$. If $\mu\approx \nu$, i.e., $\mu$ and $\nu$ are infinitesimally equivalent, then for any $\vp\in \qmd$,
\begin{equation}\label{Eq:infmaineq}
Re\iint_\md \frac{(\mu-\nu)(1-\mu\ov\nu)}{1-|\nu|^2}\vp\;dxdy
\leq\iint_\md \frac{|\mu-\nu|^2|\nu|}{1-|\nu|^2}|\vp|\;dxdy.
\end{equation}

 \begin{lemma}\label{Th:infdecr}
Let $\mu,\;\nu\in \bs$. If $\mu\approx \nu$ and $| \nu(z)|\leq |\mu(z)|$ for almost every $z\in \md$, then there is a constant $C$ depending only on $k=\lnorm{\mu}$, such that for any $\vp\in \qmd$,
\begin{equation}\label{Eq:infdecr1}
\iint_\Lambda|\alpha-\beta|^2|\vp | \;dxdy
\leq C\iint_\Lambda [|\vp|-Re(\vp sgn \mu)]\;dxdy,
\end{equation}
where $\Lambda=\{z\in \md:\; \mu(z)\neq 0\}$.
\end{lemma}
\begin{proof}
At first, let $k<1.$ Since $|\nu(z)|\leq |\mu(z)|=0$  when $z\in \md\backslash \Lambda$, it follows from  (\ref{Eq:infmaineq}) that
\begin{equation}\label{Eq:infdecr2}
-\iint_\Lambda \frac{|\mu-\nu|^2|\nu|}{1-|\nu|^2}|\vp|\;dxdy\leq Re\iint_\Lambda \frac{(\nu-\mu)(1-\mu\ov\nu)}{1-|\nu|^2}\vp\;dxdy.
\end{equation}
Hence,
\begin{equation}\label{Eq:infdecr3}
-\iint_\Lambda \frac{|\mu-\nu|^2|\mu|}{1-|\nu|^2}|\vp|\;dxdy\leq Re\iint_\Lambda \frac{(\nu-\mu)(1-\mu\ov\nu)}{1-|\nu|^2}\vp\;dxdy.
\end{equation}

In order to group $|\vp|-\vp sgn \mu$ together, we add
\[Re\iint_\md \frac{(\mu-\nu)(1-\mu\ov\nu)}{1-|\nu|^2}\frac{|\mu|}{\mu}|\vp|\;dxdy\]
to both sides of (\ref{Eq:infdecr3}) and get
\begin{equation}\label{Eq:infdecr4}
\begin{split}
&Re\iint_\md \frac{(\mu-\nu)(1-\mu\ov\nu)}{1-|\nu|^2}\frac{|\mu|}{\mu}|\vp|\;dxdy-\iint_\Lambda \frac{|\mu-\nu|^2|\nu|}{1-|\nu|^2}|\vp|\;dxdy\\
&=Re\iint_\Lambda \frac{(\nu-\mu)(1-\mu\ov\nu)}{1-|\nu|^2}\frac{|\mu|}{\mu}(|\vp|-\vp sgn \mu)\;dxdy.
\end{split}
\end{equation}
By a deformation, we have
\begin{equation}\label{Eq:infdecr5}
\begin{split}
&\iint_\Lambda\frac{(1-|\mu|^2)|\mu-\nu|^2+(1-|\mu|^2)(|\mu|^2-|\nu|^2)} {2|\mu|(1-|\nu|^2)} |\vp|\;dxdy\\
&\leq  Re\iint_\Lambda \frac{(\mu-\nu)(1-\mu\ov\nu)}{1-|\nu|^2}\frac{|\mu|}{\mu}(|\vp|-\vp sgn \mu)\;dxdy.
\end{split}
\end{equation}
Then,
\begin{equation}\label{Eq:infdecr6}
\begin{split}
&\iint_\Lambda\frac{(1-|\mu|^2)|\mu-\nu|^2} {2|\mu|(1-|\nu|^2)} |\vp|\;dxdy\\
&\leq  Re\iint_\Lambda \frac{(\mu-\nu)(1-\mu\ov\nu)}{1-|\nu|^2}\frac{|\mu|}{\mu}(|\vp|-\vp sgn \mu)\;dxdy.
\end{split}
\end{equation}
Since $|\nu(z)|\leq|\mu(z)|$, one finds that a lower bound on the coefficient of $|\mu-\nu|^2|\vp|$ on the left of (\ref{Eq:infdecr6}) is
\begin{align*}
&\frac{1-|\mu|^2}{2|\mu|(1-|\nu|^2)}\geq\frac{1-k^2}{2k}.
\end{align*}
An upper bound of the integrand on the right side of (\ref{Eq:infdecr6}) is
\[|\mu-\nu|\frac{1+|\mu|^2}{1-|\mu|^2}||\vp|-\vp sgn \mu|\leq \frac{1+k^2}{1-k^2}|\mu-\nu|\cdot||\vp|-\vp sgn \mu|.\]
Therefore,  using the identity (\ref{Eq:identity}),
we get
\begin{equation}\label{Eq:infdecr7}
\begin{split}
&\iint_\Lambda|\mu-\nu|^2 |\vp|\;dxdy\leq \frac{2k (1+k^2)}{(1-k^2)^2} \iint_\Lambda |\mu-\nu||\vp|-\vp sgn \mu|\;dxdy\\
&=\wt C\iint_\Lambda |\mu-\nu||\vp|^{\frac{1}{2}}[|\vp|-Re(\vp sgn \mu)]^{\frac{1}{2}}\;dxdy,
\end{split}
\end{equation}
where $\wt C=\wt C(k)=\frac{2\sqrt{2}k(1+k^2)}{(1-k^2)^2}.$

Applying  Schwarz's Inequality, we obtain
 \begin{align}\label{Eq:kk}
&\iint_\Lambda |\mu-\nu|^2|\vp|\;dxdy\leq \wt C^2\iint_\Lambda  [|\vp|-Re(\vp sgn \mu)]\;dxdy.
\end{align}
Now, if $k\geq 1$. Let $\mu_1=\frac{\mu}{sk},\;\nu_1=\frac{\nu}{sk}$ where $s>1$. Then $\mu_1\approx\nu_1$ and $|\nu_1|\leq |\mu_1|\leq \frac{1}{s}$ for almost all $z\in \md$. It derives from (\ref{Eq:kk}) that
\begin{align}\label{Eq:kk1}
&\iint_\Lambda |\mu-\nu|^2|\vp|\;dxdy\leq \frac{8k^2s^4(s^2+1)^2}{(s^2-1)^4}\iint_\Lambda  [|\vp|-Re(\vp sgn \mu)]\;dxdy.
\end{align}
By a refined computation, one can show that for any $k\geq0$,
\begin{equation}\label{Eq:cc}
\iint_\Lambda |\mu-\nu|^2|\vp|\;dxdy \leq 8k^2 \iint_\Lambda  [|\vp|-Re(\vp sgn \mu)]\;dxdy.
\end{equation}

\end{proof}

\section{\!\!\!\!\!{. Non-decreasable extremals of non-landslide type }}\label{S:appl}
In this section, we deform a non-Strebel unique extremal into an  extremal in a way that keeps  ``\emph{non-landslide}" and ``\emph{non-decreasable}". We need to  use the Characterization Theorem (see Theorem 1 in \cite{BLMM}) on the unique extremality. Before stating the theorem, we interpret what  the Reich's condition and Reich sequence are.

Following \cite{BLMM}, we say that $\mu\in Bel(\md)$ satisfies Reich's condition on a subset $\mcs\subset\md$ if there exists a sequence $\{\vp_n\}$  in $\qmd$ such that \\
(a) $\delta[\vp_n]:=\|\mu\|_\infty\|\vp_n\|-Re\iint_\md \mu(z)\vp_n(z)\,dxdy\to 0$, and\\
(b) $\liminf_{n\to \infty}|\vp_n(z)|>0$ for almost all $z\in \mcs$.

Generally, if $\mu$ satisfies Reich's condition above,  we call $\{\vp_n\}$ a Reich sequence for $\mu$ on $\mcs$.

The Characterization Theorem   discloses the relationship among   unique extremality infinitesimal, unique extremality and Reich's condition.
\begin{theo}\label{Th:blmm}
Let  $\mu\in M(\md)$ with a constant modulus. Then the following three conditions are equivalent:\\
(i) $\mu$ is uniquely extremal in its class in $\tmd$;\\
(ii) $\mu$ is uniquely extremal in its class in $B(\md)$;\\
(iii) $\mu$ satisfies Reich's condition on $\md$, i.e. $\mu$  has a Reich sequence on $\md$.
\end{theo}

The following theorem deforms a unique extremal Beltrami differential into a non-decreasable Beltrami differential which generally does not keep the extremality.
\begin{theorem}\label{Th:nondecr} Suppose $\eta\in M(\md)$ is uniquely extremal in $[\eta]$ and  has a constant modulus.  Let $k=\lnorm{\eta}$. Put
\begin{align*}
\mu(z)=\kappa(z)\eta(z), \;z\in \md,
\end{align*}
where $\kappa(z)$ is a non-negative measurable function on $\md$ with $\lnorm{\kappa}\leq k$.
Let $f$ be the \qc mapping from $\md$ onto itself with the Beltrami differential $\mu$. Then  $F=f^{-1}$  has a non-decreasable Beltrami differential $\wt\mu$ in its \T class $[\wt\mu]$.
\end{theorem}

\begin{proof}To avoid triviality, assume $k>0$. For any given $g\in [f]$, let $G=g^{-1}$. Let $\nu$ and $\wt\nu$ denote the Beltrami differentials of $g$ and $G$ respectively. To prove that $\wt\mu$ is non-decreasable in its \T class $[\wt\mu]$, it is sufficient to show that if $|\wt \nu(w)|\leq |\wt \mu(w)|$ holds for almost all $w\in \md$, then $\wt \mu=\wt\nu$. Use the denotations $\alpha=\wt \mu\circ f, \;\beta=\wt\nu\circ f$.

On the one hand, since $\mu$ is uniquely extremal and has constant absolute value on $\md$, by Theorem \ref{Th:blmm} $\mu$ has a Reich sequence on $\md$, that is, there is a sequence  $\{\vp_n\}\subset \qmd$ such that
\\
(a) $\delta[\vp_n]=k\|\vp_n\|-Re\iint_\md \mu(z)\vp_n(z)\,dxdy\to 0$, and\\
(b) $\liminf_{n\to \infty}|\vp_n(z)|>0$ for almost all $z\in \md$.

On the other hand, by Lemma \ref{Th:decr} and Reich's condition (a), when $|\wt \nu(w)|\leq |\wt \mu(w)|$ holds for almost all $w\in \md$, we have
\begin{equation}\label{Eq:prdecr1}
\begin{split}
&\iint_\Lambda|\alpha-\beta|^2|\vp_n | \;dxdy
\leq C\iint_\Lambda [|\vp_n|-Re(\vp_n sgn \mu)]\;dxdy\\
&=\frac{C}{k}\iint_\Lambda[k|\vp_n|-Re \mu(z)\vp_n(z)\,dxdy]\to 0,\;n\to \infty,
\end{split}
\end{equation}
where $\Lambda=\{z\in \md:\; \mu(z)\neq 0\}$ and $C$ is a constant depending only on $k$. It follows from Reich's condition (b) and Fatou's Lemma that $\alpha=\beta$ a.e. on $\Lambda$. Hence,
$\wt \mu(w)=\wt \nu(w)$ for almost all $w\in \md$.

\end{proof}

The second  theorem deforms a unique extremal Beltrami differential into a non-landslide and non-decreasable Beltrami differential which  keeps the extremality.

\begin{theorem}\label{Th:nondecr1}
 Suppose  $\eta\in M(\md)$ is uniquely extremal such that  $[\eta]$ is a not a Strebel point.  Assume in addition that $\eta$ has a constant modulus.  Let $k=\lnorm{\eta}>0$.  Suppose $E\subset\md$ is a compact subset with positive measure and empty interior. Let $\kappa$ be a non-negative measurable function on $\md$ such that $\kappa(z)=1$ for $z\in \md\backslash E$ and $\esssup_{z\in E} |\kappa|<1$.
Put
\begin{align}\label{Eq:form}
\mu(z)=\kappa(z)\eta(z), \;z\in \md,
\end{align}
Let $f$ be the \qc mapping from $\md$ onto itself with the Beltrami differential $\mu$. Then  the Beltrami differential $\wt \mu$ of $F=f^{-1}$  is  extremal, non-landslide and  non-decreasable  in its \T class $[\wt\mu]$.\end{theorem}
\begin{proof} At first,  by Theorem \ref{Th:nondecr}, $\wt \mu$ is non-decreasable in $[\wt\mu]$.
 Since $\eta$ is a non-Strebel extremal,  $\eta$    has a degenerating Hamilton sequence $\{\phi_n\}\subset Q^1(\md)$.   Noting  $\mu(z)= \eta(z)$ for  $z\in \md\backslash E$, it is easy to see that  $\{\phi_n\}$ is also a Hamilton sequence for $ \mu$ and hence $\mu$ is extremal. Furthermore, $\wt \mu$ is extremal.
It is easy to verify that $|\wt\mu(w)|=k$ for $w\in \md\backslash f(E)$ and $\esssup_{w\in f(E)}=\esssup_{z\in E} |\kappa|<1$.
 Because $f(E)$ has empty interior, by definition it is obvious that $\mu$ is non-landslide.

\end{proof}

The following two theorems are  the counterparts of Theorems \ref{Th:nondecr} and \ref{Th:nondecr1} in the infinitesimal case, respectively.
 \begin{theorem}\label{Th:infnondecr}
 Suppose $\eta\in Bel(\md)$ is infinitesimally uniquely extremal in $[\eta]$ and  has a constant modulus.  Let $k=\lnorm{\eta}$. Put
\begin{align*}
\mu(z)=\kappa(z)\eta(z), \;z\in \md,
\end{align*}
where $\kappa(z)$ is a  non-negative measurable function on $\md$ with $\lnorm{\kappa}\leq k$. Then $\mu$ is infinitesimally non-decreasable in its infinitesimal class $\emb$.
\end{theorem}

\begin{proof} When $k=0$, the proof is trivial. Now assume $k>0$. To prove that $\mu$ is non-decreasable in its infinitesimal class $\emb$, it suffices to show that, for any given $\nu\in \emb$, if  $|\nu(z)|\leq |\mu(z)|$ holds for almost all $z\in \md$, then  $\mu=\nu$.

By the Characterization Theorem (see Theorem 1 in \cite{BLMM}),  since $\mu$ is infinitesimally uniquely extremal, it
has a Reich sequence on $\md$, that is, there is a sequence  $\{\vp_n\}\subset \qmd$ such that
\\
(a) $\delta[\vp_n]=k\|\vp_n\|-Re\iint_\md \mu(z)\vp_n(z)\,dxdy\to 0$, and\\
(b) $\liminf_{n\to \infty}|\vp_n(z)|>0$ for almost all $z\in \md$.

On the other hand, by Lemma \ref{Th:infdecr} and Reich's condition (a), when $|\nu(z)|\leq |\mu(z)|$ holds for almost all $z\in \md$, we have
\begin{equation}\label{Eq:prinfdecr1}
\begin{split}
&\iint_\Lambda|\mu-\nu|^2|\vp_n | \;dxdy
\leq C\iint_\Lambda [|\vp_n|-Re(\vp_n sgn \mu)]\;dxdy\\
&=\frac{C}{k}\iint_\Lambda[k|\vp_n|-Re \mu(z)\vp_n(z)\,dxdy]\to 0,\;n\to \infty,
\end{split}
\end{equation}
where $\Lambda=\{z\in \md:\; \mu(z)\neq 0\}$ and $C$ is a constant depending only on $k$. It follows from Reich's condition (b) and Fatou's Lemma tha $\mu=\nu$ a.e. on $\Lambda$.  Hence,
$\mu(z)=\nu(z)$ for almost all $z\in \md$.

\end{proof}
\begin{theorem}\label{Th:infnondecr1}
 Suppose  $\eta\in Bel(\md)$ is uniquely extremal such that  $[\eta]_B$ is a not an infinitesimal Strebel point.  Assume in addition that $\eta$ has a constant modulus.  Let $k=\lnorm{\eta}>0$.  Suppose $E\subset\md$ is a compact subset with positive measure and empty interior. Let $\kappa$ be a non-negative measurable function on $\md$ such that $\kappa(z)=1$ for $z\in \backslash E$ and $\esssup_{z\in E} |\kappa|<1$.
Put
\begin{align*}
\mu(z)=\kappa(z)\eta(z), \;z\in \md,
\end{align*}
 Then  the Beltrami differential $ \mu$  is  extremal, non-landslide and  infinitesimally non-decreasable  in its infinitesimal class $\emb$.
 \end{theorem}
\begin{proof} At first,  by Theorem \ref{Th:infnondecr}, $ \mu$ is non-decreasable in $\emb$.
 Since $\eta$ is an infinitesimal non-Strebel extremal and   $\mu(z)= \eta(z)$ for  $z\in \md\backslash E$, it is easy to see that  $\mu$ is extremal.
Notice that $|\mu(z)|=k$ for $z\in \md\backslash E$ and $\esssup_{z\in E} |\kappa|<1$.
 Because $E$ has empty interior, by definition it is obvious that $\mu$ is infinitesimally non-landslide.

\end{proof}

\section{\!\!\!\!\!{. Construction Theorem }}\label{S:construction}

Using Theorem \ref{Th:nondecr1}, we can get extremal Beltrami differential $\wt\mu$ that is both non-landslide and non-decreasable. But it is not sure whether  $\wt\mu$ is not uniquely extremal. To ensure that $\wt\mu$ is a non-unique extremal in addition, we need to choose $E$ in Theorem \ref{Th:nondecr1} carefully.
 A 2-dimensional Cantor set $\mathscr{C}$  in $\md$ with non-zero measure is constructed for the requirement in general case.

We construct a so-called $\frac{1}{5}$-Cantor set in the closed, bounded interval $I=[0,1]$ at first. The first step in the construction is to subdivide $I$ into five intervals of equal length $\frac{1}{5}$ and remove the interior of the middle interval, that is, we remove the interval ($\frac{2}{5}$,$\frac{3}{5}$) from the interval $[0,1]$ to the obtain the closed set $C_1$, which is the union of two disjoint closed intervals, each of length $\frac{2}{5}$:
\[C_1=[0,\frac{2}{5}]\cup[\frac{3}{5},1].\]
We now repeat this ``open middle $\frac{1}{5}-$ removal" on each of the two intervals in $C_1$ to obtain a closed set $C_2$, which is the union of $2^2$ closed intervals, each of length $\frac{2^2}{5^2}$:
\[C_2=[0,\frac{4}{5^2}]\cup[\frac{6}{5^2},\frac{2}{5}]\cup[\frac{3}{5},\frac{19}{5^2}]\cup[\frac{21}{5^2},1].\]
We now repeat this ``open middle $\frac{1}{5}-$ removal" on each of the two intervals in $C_2$ to obtain a closed set $C_3$, which is the union of $2^3$ closed intervals, each of length $\frac{2^3}{5^3}$. We continue the removal operation countably many times to obtain the countable collection of sets $\{C_k\}_{k=1}^\infty$. We define the $\frac{1}{5}$-Cantor set $\mcc$ by
\[\mcc=\bigcap_{k=1}^\infty C_k.\]
The collection $\{C_k\}_{k=1}^\infty$ possesses the following properties:\\
(i) $\{C_k\}_{k=1}^\infty$ is a descending sequence of closed sets;\\
(ii) For each $k$, $C_k$ is the disjoint union of $2^k$ closed intervals, each of length $\frac{2^k}{5^k}$.

It is easy to compute the measure of $\mcc$:
\[ meas(\mcc)=1-\sum_{k=1}^\infty \frac{2^{k-1}}{5^k}=\frac{2}{3}.\]
Given $\lambda\in (0,1)$,  let $\mcc_\lambda=\lambda \mcc=:\{\lambda x, \;x\in \mcc\}$ and $\msc=\{re^{i\theta}:\; r\in \mcc_\lambda,\; \theta\in [0,2\pi)\}$.
Then $\msc$ is 2-dimensional Cantor set in $\md$ with empty interior and  $meas(\msc)>0$. \\

\noindent\textbf{Construction Theorem I.} (1) Replace $E$   by $\msc$ and keep other assumptions in Theorem \ref{Th:nondecr1}.  Then $\wt\mu$ is a non-unique extremal that is both non-landslide and non-decreasable. \\
(2) Replace $E$   by $\msc$ and keep other assumptions in Theorem \ref{Th:infnondecr1}.  Then $\mu$ is a non-unique extremal that is both non-landslide and infinitesimally non-decreasable.

\begin{proof}
By the analysis above, we only need to show that  $\mu$ is not uniquely extremal in both cases. In virtue of Theorem \ref{Th:blmm}, it is sufficient and  more convenient to prove that $\mu$ is not infinitesimally uniquely extremal.

Recall that $\mathscr{N}$ is the collection of Beltrami differentials infinitesimally equivalent to 0.  Let $\zeta\in Bel(\md)$  and define the  support set of $\zeta$ by $supp(\zeta):=\{z\in\md:\;\zeta(z)\neq 0\}$. Set
\[Z[\msc]:=\{\zeta\in \mathscr{N}:\; supp(\zeta)\subset \msc\}.\]
It is obvious that  $0\in Z[\msc]$. If  $Z[\msc]\backslash \{0\}\neq \emptyset$, then for any $\gamma\in Z[\msc]\backslash \{0\}$,    $\mu+t\gamma\in \emb$ for any $t\in \mc$. Observe the condition $\esssup_{z\in\msc}|\mu|<k$. Then, $\mu+t\gamma$ is extremal in $\emb$  when $|t|$ is sufficiently small which  implies that $\mu$ is a non-unique extremal.
It remains to show that $Z[\msc]\backslash \{0\}\neq \emptyset$. Fix a positive integer number $m$ and let
\begin{equation*}\gamma(z)=\begin{cases} z^m,\;&z\in \msc,\\
0, \;&z\in \md\backslash \msc.
\end{cases}\end{equation*}

\textit{Claim.} $\gamma\in Z[\msc]\backslash \{0\}$.

By the definition of $\mathscr{N}$, we need to show  that
\begin{equation*}
\iint_\md \gamma(z)\vp(z)\;dxdy=0,\text{ for any } \vp\in \qmd.\end{equation*}
Note that $\{1,z,z^2,\cdots, z^n,\cdots\}$ is a base of the Banach space $\qmd$. It suffices to prove
\begin{align}\label{Eq:zzn}
\iint_\md \gamma(z) z^n\;dxdy=0,\; \text{ for any } n\in \mn.
\end{align}
By the construction of $\mcc$, we see that the open set $A=[0,1]\backslash \mcc$ is the union of countably many disjoint open intervals. Let $\mathscr{A}=\lambda A:=\{\lambda z:\;z\in A\}$. Then $[0,\lambda]=\mathscr{A}\cup \mcc_\lambda$. Set $\mathscr{D}=\{re^{i\theta}:\; r\in \mathscr{A},\; \theta\in [0,2\pi)\}$. It is clear that
\[\mathscr{D}\cup \msc=\{re^{i\theta}:\; r\in [0,\lambda],\; \theta\in [0,2\pi)\}=\{z:\;|z|\leq \lambda\}.\]
Define
\begin{equation*}\wt\gamma(z):=\begin{cases}z^m, \;&z\in \mathscr{D}, \\
0,\;&z\in \md\backslash \mathscr{D},
\end{cases}\end{equation*}
and
\begin{align}\label{Eq:ga1}\Gamma(z):=\gamma(z)+\wt\gamma(z)=\begin{cases} z^m, \;&|z|\leq \lambda,\\
0,\;&\lambda<|z|<1.\end{cases}
\end{align}
A simple computation shows that
\begin{equation}\label{Eq:zzn1}
\begin{split}
&\iint_\md \Gamma(z) z^n\;dxdy=\iint_{|z|\leq \lambda} \Gamma(z) z^n\;dxdy=\iint_{|z|\leq \lambda} z^{m+n}\;dxdy\\
&=\int_0^\lambda r\;dr\int_{0}^{2\pi}e^{i(m+n)\theta}\;d\theta=0,\; \text{ for any } n\in \mn.
\end{split}\end{equation}
Observe that $\mathscr{D}$ is the union of  countably many disjoint ring domains each of which can be written in the form $R=\{re^{i\theta}:\; r\in (x,x'),\; \theta\in [0,2\pi)\}$, $x,x'\in (0,\lambda)$.
A similar computation gives
\begin{align}\label{Eq:zzn2}
\iint_R \Gamma(z) z^n\;dxdy=\iint_{R} z^{m+n}\;dxdy=0,\; \text{ for any } n\in \mn.
\end{align}
Hence, we get
\begin{align}\label{Eq:zzn3}
\iint_\mathscr{D} \Gamma(z) z^n\;dxdy=0,\; \text{ for any } n\in \mn.
\end{align}
Combining (\ref{Eq:ga1}), (\ref{Eq:zzn1}) and (\ref{Eq:zzn3}), we obtain
\begin{align}\label{Eq:zzn4}
\iint_\msc \Gamma(z) z^n\;dxdy=0,\; \text{ for any } n\in \mn,
\end{align}
which is equivalent to (\ref{Eq:zzn}). The completes the proof of Construction Theorem.

\end{proof}

Let $\vp$ be a  holomorphic function on $\md$ and  $\eta=k\frac{\ov\vp}{|\vp|}$. In one case, by the result in \cite{Lak}, for $\vp$ in a dense subset of $\qmd$, the corresponding \T differential  $\eta$ is a non-Strebel extremal (necessarily uniquely extremal).  In other case, there are a lot of holomorphic functions in $\md$  with $\iint_\md|\vp|\;dxdy=\infty$  such that $\mu$ is uniquely extremal, of course a non-Strebel extremal (see \cite{HR, Rei, Yao1}), for example, let $\vp=\frac{1}{(1-z)^2}$.

\begin{lemma}\label{Th:mat}
Let $E$ be a compact subset of $\md$, $G=\md\backslash E$ and $\vp$ a holomorphic function on $G$. Suppose that \\
\indent (a) $\mu$ is uniquely extremal on $\md$,\\
\indent (b) $\mu=\kappa(z)\frac{\ov{\vp(z)}}{|\vp(z)|}$ on $G$,\\
where $\kappa$ is non-negative measurable function on $G$. Then \\
\indent (i) $\vp$ has a holomorphic extension $\wt\vp$ from
$G$ to $\md$,
\\
\indent (ii) $\mu=k|\wt\vp|/\wt\vp$  a.e. in
$\md$ ($k=\lnorm{\mu}$).

\end{lemma}
\begin{proof} It is a simple
corollary of  Theorem G4 (the Second Removable Singularity
Theorem) of \cite{Mat}  or Theorem 2.3 on page 113  in \cite{Re2}.
\end{proof}
In the following  Construction Theorem II, we only assume that $E$ is a compact subset of $\md$ with positive measure, provided that   $\eta$ is a uniquely extremal  \T differential  representing a non-Strebel point.

\noindent\textbf{Construction Theorem II.}  Assume that $\vp$ is a holomorphic function on $\md$ such that  $\eta=k\frac{\ov\vp}{|\vp|}$ is uniquely extremal and represents a non-Strebel point. \\
(1) Keep other assumptions in Theorem \ref{Th:nondecr1}.  Then $\wt\mu$ is a non-unique extremal that is both non-landslide and non-decreasable. \\
(2) Keep other assumptions in Theorem \ref{Th:infnondecr1}.  Then $\mu$ is a non-unique extremal that is both non-landslide and infinitesimally non-decreasable.
\begin{proof}
It is sufficient to prove that $\mu$ is not uniquely extremal on $\md$. Actually, if $\mu$ is uniquely extremal, then by Lemma \ref{Th:mat}, $\mu=k\frac{\ov\vp}{|\vp|}=\eta$ on $\md$, which contradicts the assumption.
\end{proof}
The non-unique extremal $\mu$ given by Construction Theorem II is not locally extremal since otherwise by Theorem G3 (the First Removable Singularity
Theorem) of \cite{Mat}, $\mu$ is identical to $\eta$ on $\md$. However, we do not know whether the non-unique extremal $\mu$ given by Construction Theorem I is possibly locally extremal. If yes, then Problem $\mathscr{A}$ is solved.

\renewcommand\refname{\centerline{\Large{R}\normalsize{EFERENCES}}}

\end{document}